\documentclass[11pt]{amsart}
\usepackage[english]{babel}
\usepackage{amssymb}
\usepackage{amsmath}
\usepackage{srcltx}
\usepackage{lscape}

\newtheorem{dummy}{Dummy}

\newtheorem{theorem}[dummy]{Theorem}
\newtheorem{proposition}[dummy]{Proposition}
\newtheorem{corollary}[dummy]{Corollary}

\theoremstyle{definition}

\newtheorem{example}[dummy]{Example}

\newtheorem{remark}[dummy]{Remark}

\newcommand{\ignore}[1]{}

\author{S. Pumpl\"un}

\email{susanne.pumpluen@nottingham.ac.uk}
\address{School of Mathematical Sciences\\
University of Nottingham\\
University Park\\
Nottingham NG7 2RD\\
United Kingdom
}


\begin{document}

\title[The multiplicative loops of Jha-Johnson semifields]
{The multiplicative loops of Jha-Johnson semifields}

\begin{abstract}
The multiplicative loops  of  Jha-Johnson semifields are
 non-automorphic finite loops whose left and right nuclei are the multiplicative groups of a field extension
of their centers. They yield examples of finite loops with non-trivial automorphism group and non-trivial
 inner mappings. Upper bounds are given for the number
 of non-isotopic multiplicative loops of order $q^{nm}-1$ that are defined using the twisted polynomial ring
 $K[t;\sigma]$ and a twisted irreducible polynomial of degree $m$, when the automorphism $\sigma$ has order $n$.
\end{abstract}

\maketitle

\section*{Introduction}

 For loops the role of automorphisms  is not the same as for groups, as inner mappings need not be automorphisms, and
almost all finite loops have a trivial automorphism group \cite{MW}.
 One systematic way to find finite loops with a non-trivial automorphism group is to find semifields
 (i.e., finite unital nonassociative division algebras) with a non-trivial
 automorphism group, as their  multiplicative loops clearly will also have  non-trivial automorphisms.
 The automorphism groups of semifields, however,
 have not been studied in great detail so far, mostly because it is usually important to understand semifields only up to
  isotopy. The multiplicative loops and multiplication groups of semifields were considered for instance in \cite{CJ, NG, NG2}.

We look at the  multiplicative loops of
a family of Jha-Johnson semifields, which include Sandler and Hughes-Kleinfeld semifields as special cases.
These multiplicative loops yield examples of finite loops with both non-trivial automorphism groups and non-trivial
inner automorphisms, which
 form non-trivial subgroups of their inner mapping groups.

 For their definition, let $K=\mathbb{F}_{q^n}$ be a finite field, $\sigma$  an automorphism of $K$ of order $n$ with fixed field
$F=\mathbb{F}_{q}$,  and $R = K[t;\sigma]$ a twisted polynomial ring. Given an irreducible monic
twisted polynomial $f\in R = K[t;\sigma]$ of degree $m\geq 2$ that is not right-invariant
(i.e., $ Rf$ is not a two-sided ideal),
the set of all twisted polynomials in $R$ of degree less than $m$, together with the usual addition of polynomials, can be equipped with a nonassociative
ring structure using right division by $f$ to define the multiplication as $g\circ h=gh \,\,{\rm mod}_r f $,
employing the remainder of dividing by $f$ on the right.
The resulting nonassociative unital ring $S_f$, also denoted by $K[t;\sigma]/K[t;\sigma]f$,
is a proper semifield over its center $F$ of order  $q^{mn}$ and is an example of a \emph{Jha-Johnson semifield}. Its left and middle nucleus is
$K= \mathbb{F}_{q^n}$ and its right nucleus is isomorphic to $\mathbb{F}_{q^m}.$
 In particular, \emph{Sandler semifields} \cite{San62} are obtained by choosing $n\geq m$ and
 $f(t) = t^m - a\in K[t; \sigma]$ with $a \in K \setminus F$,
 and \emph{Hughes-Kleinfeld semifields}  by choosing $f(t)= t^2-a_1 t-a_0$ \cite{HK}, \cite{Kn65}.

This nonassociative algebra construction was initially introduced in a more general setup by Petit \cite{P66}. It was used for the first time
for semifields by Wene  \cite{W00} and then more recently  by Lavrauw and Sheekey \cite{LS}.

The structure of the paper is as follows: in Section \ref{sec:prel},
we introduce the basic terminology and  the algebras $S_f$.
We then determine some subgroups of automorphism groups of the multiplicative loops $L_f=S_f\setminus \{0\}$ of semifields $S_f$ in Section 2
by restricting
the automorphisms of $S_f$  to $L_f $. We give examples of loops of order $q^{mn}-1$ with multiplicative subgroups of order $m$.
The inner automorphisms of the loops $L_f $ are treated in Section \ref{sec:inner} and
certain types of Sandler semifields (also called \emph{nonassociative  cyclic algebras} due to the similarity in their construction and behaviour
 to associative
cyclic algebras) are studied in Section \ref{sec:nonasscyclic}:
their loops have nontrivial automorphism groups containing a cyclic subgroup of order $s =(q^m-1)/(q-1)$
of inner automorphisms, i.e. middle inner mappings. Upper bounds for the number of isotopic loops obtained as
multiplicative loops of Jha-Johnson semifields are briefly considered in Section \ref{sec:isotopies}.

Since the multiplication tables of finite loops yield regular Latin squares,
 all results  can be carried over to regular Latin squares if desired.

\section{Preliminaries} \label{sec:prel}


\subsection{Nonassociative algebras} \label{subsec:nonassalgs}


Let $F$ be a field and let $A$ be an $F$-vector space. $A$ is an
\emph{algebra} over $F$ if there exists an $F$-bilinear map $A\times
A\to A$, $(x,y) \mapsto x \cdot y$, denoted  by juxtaposition
$xy$, the  \emph{multiplication} of $A$. We will only consider unital algebras which have a unit element
 denoted by 1, such that $1x=x1=x$ for all $x\in A$.

The {\it associator} of $A$ is given by $[x, y, z] =
(xy) z - x (yz)$. The {\it left nucleus} of $A$ is defined as ${\rm
Nuc}_l(A) = \{ x \in A \, \vert \, [x, A, A]  = 0 \}$, the {\it
middle nucleus} of $A$ as ${\rm Nuc}_m(A) = \{ x \in A \, \vert \,
[A, x, A]  = 0 \}$ and  the {\it right nucleus} of $A$ as
${\rm Nuc}_r(A) = \{ x \in A \, \vert \, [A,A, x]  = 0 \}$. ${\rm Nuc}_l(A)$, ${\rm Nuc}_m(A)$, and ${\rm Nuc}_r(A)$ are associative
subalgebras of $A$. Their intersection
 ${\rm Nuc}(A) = \{ x \in A \, \vert \, [x, A, A] = [A, x, A] = [A,A, x] = 0 \}$ is the {\it nucleus} of $A$.
${\rm Nuc}(A)$ is an associative subalgebra of $A$ containing $F1$
and $x(yz) = (xy) z$ whenever one of the elements $x, y, z$ is in
${\rm Nuc}(A)$.   The
 {\it center} of $A$ is ${\rm C}(A)=\{x\in A\,|\, x\in \text{Nuc}(A) \text{ and }xy=yx \text{ for all }y\in A\}$.

An algebra $A$ is called a \emph{division algebra} if for any
$a\in A$, $a\not=0$, the left multiplication  with $a$, $L_a(x)=ax$,
and the right multiplication with $a$, $R_a(x)=xa$, are bijective.
 If $A$ has finite dimension over $F$, then $A$ is a division algebra if
and only if $A$ has no zero divisors \cite[ pp. 15, 16]{Sch}.
An element $0\not=a\in A$ has a \emph{left inverse} $a_l\in A$, if
$R_a(a_l)=a_l a=1$, and a \emph{right inverse}
 $a_r\in A$, if $L_a(a_r)=a a_r=1$.

 A \emph{semifield} is a finite-dimensional unital division algebra over a finite field.
 It is called \emph{proper} if it is not associative.

An automorphism $G\in {\rm Aut}_F(A)$ is an \emph{inner automorphism}
if there is an element $m\in A$ with left inverse $m_l$ such
that $G(x)= (m_lx)m$ for all $x\in A$. We denote this automorphism by $ G_m$.
The set of inner automorphisms $\{G_m\,|\, m\in {\rm Nuc}(A) \text{ invertible} \}$ is a subgroup of ${\rm Aut}_F(A)$
\cite[Lemma 2, Theorem 3]{W09}.

\subsection{Loops}\label{subsec:loops}

A \emph{loop} $L$ is a quasigroup with a unit element, i.e. a set $L$ together with a multiplicative structure
$\cdot:L\times L\longrightarrow L$ such that two of the elements $x,y,z\in L$ in
$x\cdot y=z$ uniquely  determine the third one.
The left and right multiplication maps are the bijections $L_a:x\mapsto ax$, $R_a:x\mapsto xa$ on $L$.
The \emph{multiplication group} ${\rm Mlt}(L)$ is the permutation group generated by the left and right multiplication maps of $L$.

We define the middle,  left and right  inner  mappings on $L$ as
 $T_x=L_x^{-1}R_x$, $L_{x,y}=L_{yx}^{-1}L_yL_x$ and $R_{x,y}=R_{xy}^{-1}R_yR_x$.  The inner mapping $T_x$ plays the role of conjugation and measures how commutative the loop is, and the mappings
$L_{x,y}$ and $R_{x,y}$ measure deviations from associativity.
The \emph{ inner mapping group} is defined as ${\rm Inn}(L)=\{f\in {\rm Mlt}(L)\,|\, f(1)=1\}$ and it is well-known that
${\rm Inn}(L)=\langle L_{x,y}, R_{x,y},T_x\,|\, x,y\in L\rangle$.

The set $A^\times$ of non-zero elements of any division algebra $A$ is a multiplicative loop. Thus examples of
finite loops canonically appear as the multiplicative loops of semifields.
The center, left, right and middle nucleus of the multiplicative loop of a semifield
(defined analogously as for an algebra) are cyclic groups.
If $A$ is a semifield with center $F\cdot 1=\mathbb{F}_q\cdot 1$ and
 $q^r$ elements, then its
multiplicative loop has $q^r-1$ elements and for $r\geq 3$, the multiplication group ${\rm Mlt}(A^\times)$
 is a transitive subgroup of  ${\rm GL}(r,q)$ \cite[Proposition 5.1]{NG2}.
 Furthermore,
$${\rm SL}(r,q)\leq {\rm Mlt}(A^\times)\leq {\rm GL} (r,q)$$
(\cite[Proposition 5.2]{NG2} or \cite[Proposition 2.2]{NG2}), so that
$$q^{r(r-1)/2}(q^r-1)(q^{r-1}-1)\cdots(q^2-1)\leq| {\rm Mlt}(A^\times)|\leq q^{r(r-1)/2}(q^r-1)(q^{r-1}-1)\cdots(q-1).$$
In particular, for $q=2$, we get
$|{\rm Mlt}(A^\times)|=2^{r(r-1)/2}(2^r-1)(2^{r-1}-1)\cdots(2^2-1).$

For a semifield $A$,
$A^\times$ is  \emph{left cyclic} (or \emph{left primitive}) if
$A^\times=\{a,a^{(2},a^{(3},\dots, a^{(s-1}\}$
 where the left  principal
powers $a^{(n}$ of $a$ are defined as $a^{(1}=a$, $a^{(n}=a^{(n-1}a$, and
\emph{right cyclic} (or \emph{right primitive}) if
$A^\times=\{a,a^{2)},a^{3)},\dots, a^{s-1)}\}$
 where the the right principal
powers $a^{n)}$ of $a$ are defined as $a^{1)}=a$, $a^{n)}=aa^{n-1)}$.
For instance, the multiplicative loops of semifields of order 16, 27, 32, 125 and 343 are all right cyclic.
Semifields of order 81 are left and right cyclic \cite{HR}.

A loop $L$ has the \emph{weak Lagrange property} if for each subloop $M$ of $L$, $|M|$ divides $|L|$, and the
\emph{strong Lagrange property}
if every subloop $M$ of $L$ has the weak Lagrange property.
(A loop may have the weak Lagrange property but not the strong one.)

In general, the structure of the multiplicative loop of a semifield is not known.


\subsection{Semifields obtained from skew polynomial rings} \label{subsec:structure}


Let $K$ be a field and $\sigma$ an automorphism of $K$. The \emph{twisted polynomial ring} $R=K[t;\sigma]$
is the set of polynomials $a_0+a_1t+\dots +a_nt^n$ with $a_i\in K$,
where addition is defined term-wise and multiplication by
$ta=\sigma(a)t $ for all $a\in K$ \cite{O1}.
For $f=a_0+a_1t+\dots +a_mt^m$ with $a_m\not=0$ define ${\rm
deg}(f)=m$ and put ${\rm deg}(0)=-\infty$. Then ${\rm deg}(fg)={\rm deg}
(f)+{\rm deg}(g).$
 An element $f\in R$ is \emph{irreducible} in $R$ if it is not a unit and  it has no proper factors,
 i.e if there do not exist $g,h\in R$ with
 ${\rm deg}(g),{\rm deg} (h)<{\rm deg}(f)$ such
 that $f=gh$.

 $R=K[t;\sigma]$ is a left and right principal ideal domain
 and there is a right division algorithm in $R$: for all
$g,f\in R$, $f\not=0$, there exist unique $r,q\in R$ with ${\rm
deg}(r)<{\rm deg}(f)$, such that $g=qf+r$ \cite{J96}.

Let $f\in R=K[t;\sigma]$ have degree $m\geq 1$. Let
${\rm mod}_r f$ denote the remainder of right division by $f$. Then the additive abelian group
$$R_m=\{g\in K[t;\sigma]\,|\, {\rm deg}(g)<m\}$$
 together with the multiplication
 $g\circ h=gh \,\,{\rm mod}_r f $
 is a unital nonassociative algebra $S_f=(R_m,\circ)$ over
 $F=\{a\in K\,|\, ah=ha \text{ for all } h\in S_f\}$,
which is a subfield of $K$ \cite[(7)]{P66}.
  $S_f$ is also denoted by $R/Rf$
 if we want to make clear which ring $R$ is involved in the construction.
  In the following, we  denote the multiplication  by  juxtaposition.
  Note that when ${\rm deg}(g){\rm deg}(h)<m$, the multiplication of $g$ and $h$ in $S_f$ and hence in $L_f$
 is the same as the multiplication of $g$ and $h$ in $R$ \cite[(10)]{P66}.

We  assume throughout the paper that
$$f(t) = t^m -\sum_{i=0}^{m-1} a_i t^i \in K[t;\sigma]$$
has degree $ m \geq 2$, since if $f$ has degree $1$ then  $S_f\cong K$. We also assume that $\sigma\not=id$.
W.l.o.g. we may consider only monic $f(t)$, since $S_f= S_{af}$ for all $a\in K^\times$.

From now on,
$$K = \mathbb{F}_{q^n}$$
is a finite field,  $q = p^r$ for some prime $p$, $\sigma$ is an automorphism of $K$ of order $n>1$ and thus
$$F={\rm Fix}(\sigma)=\mathbb{F}_{q},$$
 i.e. $K/F$ is a cyclic Galois extension of degree $n$
with ${\rm Gal}(K/F)=\langle\sigma\rangle.$  The norm  $N_{K / F} : K^{\times} \rightarrow F^{\times}$ is surjective, and
${\rm ker}(N_{K/F})$ is a cyclic group of order
$s =(q^n-1)/(q-1).$
Then $S_f$ is a semifield if and only if $f$ is irreducible, and a proper semifield
 if and only if $f$ is not right-invariant (i.e., the left ideal $Rf$ generated by $f$ is not two-sided),
 cf. \cite[(2), p.~13-03, (5), (9)]{P66}, or \cite{LS}.
 We will hence assume from now on that $f$ is not
 right-invariant and irreducible. Then $S_f$ is a \emph{ Jha-Johnson semifield}.
  Moreover, each Jha-Johnson semifield is isotopic to some such algebra $S_f$ \cite[Theorem 16]{LS}, so in particular,
  each multiplicative loop of a Jha-Johnson semifield is isotopic to the  multiplicative loop of some algebra $S_f$.

   Let $L_f=S_f\setminus \{0\}$
  denote the multiplicative loop of the proper semifield $S_f$. Then
  $$|L_f|=q^{mn}-1, \quad {\rm Nuc}_l(L_f)={\rm Nuc}_m(L_f)=K^\times= \mathbb{F}_{q^n}^\times$$
 and ${\rm Nuc}_r(S_f)=\{g\in R_m\,|\, fg\in Rf\}\cong\mathbb{F}_{q^m}$
 implies
 ${\rm Nuc}_r(L_f)\cong\mathbb{F}_{q^m}^\times.$

\begin{proposition} (\cite[Proposition 3]{BP})\label{prop:subfields}
 Let $f(t) \in F[t]=F[t;\sigma]\subset K[t;\sigma]$ be monic, irreducible and not right-invariant.
 Then $ {\rm Nuc}_r(L_f)\cong F[t]/(f(t))^\times.$
In particular, ${\rm Nuc}(L_f)=F^\times\cdot 1$ and  $ft\in Rf$.
\end{proposition}

\begin{remark} \label{re:t}
(i) The powers of $t$ are associative if and only if $t^mt=tt^m$ in $S_f$
 if and only if $t\in {\rm Nuc}_r(S_f)$ if and only if $ft\in Rf$ \cite[(5)]{P66}.
  Note that the associativity of the powers of $t$ does not imply that $S_f$ is a field.
 \\ (ii) The powers of $t$ form a multiplicative group in $S_f$ if and only if  $f(t) \in F[t]$ \cite[(16) or (20)]{P66}.
\\ (iii)
Note that  $f(t)\in K[t;\sigma]\setminus F[t;\sigma]$ is never
right-invariant and that if $f(t)\in  F[t]$ has degree $m<n$, then $f(t)\in K[t;\sigma]$ is also
never right-invariant in $K[t;\sigma]$.
 For $n=m$ the only right-invariant $f(t)\in  F[t]$ are of the form $f(t)=t^m-a$,
 and these polynomials are reducible. Hence all irreducible polynomials in $F[t]$ of degree $n$ are not right-invariant
  in $K[t;\sigma]$.
\end{remark}

By \cite[Lemma 10]{Ki}, we have
$$|{\rm Mlt}(L_f)|=|L_f||{\rm Inn}(L_f)|=(q^{mn}-1)|{\rm Inn}(L_f)|.$$
We will focus on those Jha-Johnson semifields which are algebras $S_f$, and apply the results from \cite{BP}.
We  concentrate on the case that $n\geq m-1$. For $n < m-1$ similar but weaker results can be obtained.

%
%

\section{The automorphisms of the multiplicative loops of $S_f$}

Let $f(t) = t^m -\sum_{i=0}^{m-1} a_i t^i \in K[t;\sigma]$ be irreducible and not right-invariant.
Let $L_f$ be the multiplicative loop of $S_f$.
We obtain the following necessary conditions for a loop automorphism of $L_f$:

\begin{proposition}
Let $H\in {\rm Aut}(L_f)$ be a loop automorphism.
\\ (i) $H \vert_{K^\times} = \tau$ for some $\tau \in \text{Aut}(K^\times)$.
\\ (ii) If $\sigma$ has order $n\geq m-1$ then $H(t) = kt$ for some $k \in K^{\times}$.
\\ (iii) If $\sigma$ has order $n\geq m-1$ then for all $i\in \{ 1, \ldots, m-1 \}$ and $z \in K^\times$,
\begin{equation*}
H(z t^i) = H(z)H(t)^i = \tau(z) (kt)^i = \tau(z) \Big(\prod_{l=0}^{i-1} \sigma^l(k) \Big) t^i.
 \end{equation*}
\\ (iv) If $\sigma$ has order $n\geq m-1$ then
$$H(t^i) = H(t)^i = (kt)^i =  \Big(\prod_{l=0}^{i-1} \sigma^l(k) \Big) t^i$$
for all $i\in \{ 1, \ldots, m\}$. In particular,
$H \Big( \sum_{i=0}^{m-1} a_i t^i \Big) =k \sigma(k) \cdots \sigma^{m-1}(k)\sum_{i=0}^{m-1} a_i t^i.$
\end{proposition}

\begin{proof}
(i) Since $L_f$ is not
associative, $\text{Nuc}_l(L) = K^\times$. Since any automorphism  preserves the left nucleus,
$H(K^\times) = K^\times$ and so $H \vert_{K^\times} = \tau$ for some
 $\tau \in \text{Aut}(K^\times)$.
 \\ (ii)  Suppose $H(t) = \sum_{i=0}^{m-1} k_i t^i$ for some $k_i \in K$.
Then we have
\begin{equation} \label{eqn:automorphism_necessity_theorem 1}
H(tz) = H(t)H(z) = (\sum_{i=0}^{m-1} k_i t^i ) \tau(z) =
\sum_{i=0}^{m-1} k_i \sigma^{i}(\tau(z)) t^i \end{equation} and
\begin{equation} \label{eqn:automorphism_necessity_theorem 2} H(tz) =
H(\sigma(z)t) =H(\sigma(z))H(t)= \sum_{i=0}^{m-1} \tau(\sigma(z)) k_i t^i
\end{equation} for all $z \in K^\times$. Comparing the coefficients of $t^i$
in \eqref{eqn:automorphism_necessity_theorem 1} and
\eqref{eqn:automorphism_necessity_theorem 2} we obtain
\begin{equation} \label{eqn:automorphism_necessity_theorem 3} k_i \sigma^{i}(\tau(z)) =
k_i \tau(\sigma^i(z))= \tau(\sigma(z)) k_i=k_i\tau(\sigma(z))
\end{equation} for all $i \in\{ 0, \ldots, m-1\}$ and all $z \in K^\times$. This implies
$k_i (\tau( \sigma^i (z)) - \tau(\sigma(z)) )=0$
for all $i \in\{ 0, \ldots, m-1 \}$ and all $z \in K$ since $\sigma$
and $\tau$ commute, i.e. $k_i=0$ or $\tau( \sigma^i (z)) = \tau(\sigma(z)) $, that means (since $\tau$ is bijective)
\begin{equation} \label{eqn:automorphism_necessity_theorem 4}
k_i=0 \text{ or } \sigma^{i}(z)=\sigma(z) \end{equation} for all $i
\in\{ 0, \ldots, m-1\}$ and all $z \in K^\times$.

Since $\sigma$ has order $n\geq m-1$,  $\sigma^i\not=\sigma$ for
all $i\in\{ 0, \ldots, m-1\}$, $i\not=1$,  so
\eqref{eqn:automorphism_necessity_theorem 4} implies $k_i = 0$ for
all $ i \in\{ 0, \ldots, m-1\}$, $i\not=1$. Therefore $H(t) = kt$ for some
$k \in K^{\times}$.
\\ (iii)  For all $i\in \{ 1, \ldots, m-1 \}$ and $z \in K$,
\begin{equation*}
H(z t^i) = H(z)H(t)^i = \tau(z) (kt)^i = \tau(z) \Big(\prod_{l=0}^{i-1} \sigma^l(k) \Big) t^i. \end{equation*}
\\ (iv) By (iii) we have
$H(t^i) = H(t)^i = (kt)^i =  \Big(\prod_{l=0}^{i-1} \sigma^l(k) \Big) t^i$
for all $i\in \{ 1, \ldots, m-1 \}$. Moreover, we know that $t^m=t t^{m-1}$ and
$H(t t^{m-1})=H(t)H(t^{m-1})=H(t)H(t)^{m-1}$, i.e.
$$H(t)^m =H(t)H(t)^{m-1}= k \sigma(k) \cdots \sigma^{m-1}(k) t^{m} =
k \sigma(k) \cdots \sigma^{m-1}(k)\sum_{i=0}^{m-1} a_i t^i.$$
This gives
$H(t)^m =H(t^m) = H \Big( \sum_{i=0}^{m-1} a_i t^i \Big) =
k \sigma(k) \cdots \sigma^{m-1}(k)\sum_{i=0}^{m-1} a_i t^i.$
We obtain
$$H \Big( \sum_{i=0}^{m-1} a_i t^i \Big) =k \sigma(k) \cdots \sigma^{m-1}(k)\sum_{i=0}^{m-1} a_i t^i.$$
or in other words,
$H (t^m) =k \sigma(k) \cdots \sigma^{m-1}(k) t^m.$
\end{proof}

A closer look at the proofs of \cite[Theorem 4, 5]{BP} shows that these  can be generalized
to obtain the following statement:

 \begin{theorem} \label{thm:automorphism_of_Sf_field_case}
(i) Let  $n\geq m-1$. Then  a map $H:S_f\longrightarrow S_f$ is an
automorphism of the ring $S_f$ if and only if $H=H_{\tau , k}$ where   $\tau\in {\rm Aut}(K)$ and
$$H_{\tau , k}(\sum_{i=0}^{m-1} x_i t^i )= \tau(x_0) + \tau(x_1)kt + \tau(x_2)k\sigma(k)t^2+\cdots +
\tau(x_{m-1}) k \sigma(k) \cdots \sigma^{m-2}(k) t^{m-1},$$
with $k \in K^{\times}$ such that
\begin{equation} \label{eqn:neccessary}
\tau(a_i) = \Big( \prod_{l=i}^{m-1}\sigma^l(k) \Big) a_i
\end{equation}
for all $i \in\{ 0, \ldots, m-1\}$.
\\ (ii) Let  $n< m-1$. For all $k\in K^\times$ satisfying Equation (\ref{eqn:neccessary}) the maps $H_{\tau , k}$
from (i) are automorphisms of the ring $S_f$ and form a subgroup of ${\rm Aut}(S_f)$.
\end{theorem}

 \begin{corollary} \label{cor:automorphism_of_Lf_field_case}
Let $\tau\in {\rm Aut}(K) $  and $k \in K^{\times}$ satisfy  (\ref{eqn:neccessary})
for all $i \in\{ 0, \ldots, m-1\}$.
Then
$$H_{\tau , k}(\sum_{i=0}^{m-1} x_i t^i )= \tau(x_0) + \tau(x_1)kt + \tau(x_2)k\sigma(k)t^2+\cdots +
\tau(x_{m-1}) k \sigma(k) \cdots \sigma^{m-2}(k) t^{m-1}$$
is an automorphism of the multiplicative loop $L_f$ which extends to an automorphism of the ring
$S_f$. For all $k\in K^\times$ satisfying  (\ref{eqn:neccessary}) and $\tau\in {\rm Aut}(K) $, the
 automorphisms $H_{\tau , k}$ form a subgroup of ${\rm Aut}(L_f)$.
\end{corollary}

\begin{remark} (i)
 Let  $n \geq m-1$. For all $\tau\in {\rm Aut}_F(K^\times)$ and $k \in K^{\times}$  such that
$
\tau(a_i) = \Big( \prod_{l=i}^{m-1}\sigma^l(k) \Big) a_i
$
for all $i \in\{ 0, \ldots, m-1\}$, a tedious calculation shows  that
\begin{equation} \label{eqn:automain}
H_{\tau , k}(\sum_{i=0}^{m-1} x_i t^i )=  \tau(x_0) + \sum_{i=1}^{m-1} \tau(x_i)
\big(\prod_{l=0}^{i-1}\sigma^l(k)\big) t^i,
\end{equation}
is only  an automorphism of $L_f$, if $\tau$ is additive. So if there are automorphisms of $L_f$
not induced by automorphisms of $S_f$, they will not have the form $H_{\tau , k}$.
\\ (ii)
 Let $G$ be a ring automorphism of  $R=K[t;\sigma]$. Then for $h(t) = \sum_{i=0}^{r} b_i t^i \in R$ we have
$$G(h(t))=
\tau(b_0)  +\sum_{i=i}^{m-1}\tau(b_i) \prod_{l=0}^{i-1}\sigma^l(k) t^i$$
for some $\tau\in{\rm Aut}(K)$
and some $k\in K^\times$
(\cite[Lemma 1]{LS}, or cf. \cite[p.~75]{Kish}).
The automorphisms $H_{\tau,k}$ of ${\rm Aut}(L_f)$ with $\tau\in {\rm Aut}(K)$ a field automorphism are canonically induced by the automorphisms $G$
of $R$ which satisfy  \eqref{eqn:neccessary}:
It is straightforward to see that
$S_f\cong S_{G(f)}$ \cite[Theorem 7]{LS} and consequently, $L_f\cong L_{G(f)}.$
In particular, if $k\in K^\times$ satisfies Equation \eqref{eqn:neccessary} then
$$G(f(t)) = \big( \prod_{l=0}^{m-1} \sigma^l(k) \big) f(t),$$
thus $G$ induces an automorphism of $L_f$.
\end{remark}

From \cite[Theorem 4]{BP} we obtain analogously as in the proofs for  \cite[Theorems 7, 8]{BP} that
 the subgroups   of  ${\rm Aut}(L_f)$
 that are induced by automorphisms of $S_f$ restricted to $L_f$,
 all  are subgroups induced by  the automorphisms of Sandler semifields  for certain $f$:

\begin{proposition} \label{Aut(S_f) subgroup corollary}
 Let
 $g(t) = t^m - \sum_{i=0}^{m-1} b_i t^i \in K[t;\sigma]$
  be irreducible and not right-invariant.  Let
  $$f(t) = t^m - b_0 \in K[t;\sigma] \text{ such that } b_0 \in K \setminus F$$ or more generally,
 $$f(t) = t^m - \sum_{i=0}^{m-1} a_i t^i \in K[t;\sigma],$$
where $a_i \in \left\{ 0 , b_i \right\}$ for all $i \in \{ 0,\ldots , m-1\}$,
 be  irreducible and not right-invariant. Then
$$\{H\in {\rm Aut}(L_g)\,|\, H \text{ extends to some } H_{\tau , k} \in {\rm Aut}(S_g) \}$$
 is a subgroup of
$$\{H\in {\rm Aut}(L_f)\,|\, H \text{ extends to some }H_{\tau , k} \in {\rm Aut}(S_f) \}.$$
\end{proposition}

\begin{proposition} \label{prop:fixedfield}
Suppose
 $f(t) = t^m - \sum_{i=0}^{m-1} a_it^i \in F[t]\subset K[t;\sigma]$
 is irreducible and not right-invariant.
 \\ (i) (\cite[Theorem 11]{BP})
 $\langle H_{\sigma,1}\rangle \cong \mathbb{Z}/n \mathbb{Z}$ is a cyclic subgroup of ${\rm Aut}(L_f)$.
 \\ (ii) (\cite[Corollary 20]{BP})
 Suppose $a_{m-1}\in F^\times$. Then for all
$\tau\in {\rm Aut}(K) $, the maps  $H_{\tau , 1}$  are ring automorphisms of $S_f$, and thus
 ${\rm Aut}(K)$ is isomorphic to a subgroup of  ${\rm Aut}(L_f)$.
 \\ (iii) (Remark \ref{re:t}) The powers of $t$ form a multiplicative group of order $m$ in the loop $L_f$
 which has order $q^{mn}-1$.
\end{proposition}

\begin{example}\label{ex:I}
Write $K = \mathbb{F}_{p^l},$ so that  with $K = \mathbb{F}_{q^n}$, $q = p^r$ we have $l = nr$. For $r\not=0$ choose
 $$\sigma(x)=x^{p^r},$$
i.e. ${\rm Fix}(\sigma)=\mathbb{F}_{p^r}=F$ and $\sigma$ has order $n$. Let $K^\times=\langle \alpha\rangle$ and
$m$ be a prime divisor of $p^{gcd(r,l)}-1=p^r-1$. Then there is some $a\in K^\times$ such that
$K[t;\sigma]/K[t;\sigma](t^m-a)$
 is a proper semifield if and only if
$$gcd\big((p^l-1)(p^r-1),(p^{mr}-1)\big)>p^r-1.$$
These $a\in K^\times$ are given by $a=\alpha^u$ for any $u\not\in\mathbb{Z}(p^{mr}-1)(p^r-1)^{-1}$.

For $m=2,3,$ the assumption that $m$ divides $p^r-1$ is not needed.
Note that for
$$(p,m)\in\{(11,5),(29,7),(67,11),\dots\}$$
where
$p=1\text{ mod } m$ and where $l=nr$ is a multiple of $m$, sufficient conditions are easier to find (\cite[(22)]{P66}).
This way we obtain loops of order  $|K|^m-1=p^{mnr}-1$ with left and middle nucleus
$\mathbb{F}_{p^l}\cdot 1$ and center $\mathbb{F}_{p^r}\cdot 1$.
For $n\geq m$ these are the multiplicative loops of Sandler semifields.
 Using the results on the automorphisms of $S_f$ for $f(t)=t^m-a$ in \cite{BP}
(cf. also \cite[Corollary 5.32]{CB}) or simply a straightforward calculation, we can conclude the following: Let $s=(p^{rm}-1)/(p^r-1)$. If $\omega$ is an
$s$th root of unity in $K$, then $H_{id,\omega}\in {\rm Aut}(S_f)$ and the subgroup of automorphisms
 $\langle H_{id,\omega}\rangle$ is isomorphic to the cyclic subgroup in
$K^\times$ of all the $s$th roots of unity in $K$, which has
$$S(r,m,l)=gcd\big( \frac{p^{rm}-1}{p^r-1},p^l-1\big)$$
elements. Therefore there are at least $S(r,m,l)$
distinct automorphisms of $L_f$ of the form $H_{id,k}$ with $k$ an $s$th root of unity,
which form a subgroup of ${\rm Aut}(L_f)$. In particular we have:
\\ (i) If $p\equiv 1 \text{ mod }m$ then $S(r,m,l)\geq m$ and so there are at least $m$ automorphisms of $L_f$ of the form $H_{id,k}$,
which form a subgroup of ${\rm Aut}(L_f)$.
\\ (ii) Suppose $l$ is even and at least one of $r,m$ is also even. Then if
 $p\equiv -1 \text{ mod }m$, there are at least $m$ automorphisms of $L_f$ of the form $H_{id,k}$,
 which form a subgroup of ${\rm Aut}(L_f)$.
 \\ (iii) Suppose $p$ is odd and $m=2$. Then ${\rm Aut}(L_f)$ is not trivial by (i).
\end{example}

\begin{example}
 Suppose $K = \mathbb{F}_{11^2}$ and $\sigma(x)=x^{11}.$ Then $\sigma$ has order $n=2$.
 Write $K^\times=\langle \alpha\rangle$ and choose $a=\alpha^{12}$, then for $f(t)=t^5-a$,
 $$S_f=K[t;\sigma]/K[t;\sigma](t^5-a)$$
 is a proper semifield over
 $F=\mathbb{F}_{11}$, where the powers of $t$
 form a multiplicative group  \cite[(22)]{P66}.
Its multiplicative loop  has order $11^{10}-1$, left and middle nucleus $(\mathbb{F}_{11^2})^\times\cdot 1$ and
right nucleus isomorphic to
$\big(\mathbb{F}_{11}[t]/(t^5-a)\big)^\times\cong (\mathbb{F}_{11^5})^\times.$
Thus its nucleus is $(\mathbb{F}_{11})^\times \cdot 1$ and equals its center.
It has a subgroup  of order $5$ generated by $t$ and  is clearly not weak Lagrangian.
Its automorphism group
 contains $\langle H_{\sigma,1}\rangle\cong\mathbb{Z}/2\mathbb{Z}$ as a subgroup  (Proposition \ref{prop:fixedfield}).
It also contains at least 5 automorphisms of the form $H_{id,k}$ by Example \ref{ex:I} (iii).
Since the nucleus of $L_f$ equals its center, there are no non-trivial inner mappings of $L_f$ induced by inner automorphisms of $S_f$
and
$$11^{45}(11^{9}-1)\cdots(11^2-1) \leq| {\rm Inn}(L_f)|\leq
11^{45}(11^{9}-1)\cdots (11^2-1)  10;$$
 see the next Section.
\end{example}

\begin{example}
Let $f(t)=t^2-a_1t-a_0\in K[t;\sigma]$ be irreducible  and not right-invariant. Then $S_f$ is a \emph{Hughes-Kleinfeld semifield}.
 $L_f$ is a loop of order $q^{2n}-1$ with left and middle nucleus $K^\times\cdot 1=\mathbb{F}_{q^n}^\times\cdot 1$ and
 right nucleus
$\mathbb{F}_{q^2}^\times \cdot 1$.
For $a_0,a_1\in F^\times$,  ${\rm Aut}(S_f)\cong \mathbb{Z}/n\mathbb{Z}$, so ${\rm Aut}(L_f)$ contains a cyclic subgroup of order $n$.
(If $a_0\in K\setminus F$ and $a_1\in F^\times$ then ${\rm Aut}(S_f)$ is trivial.)
We know that $f(t)=t^2-a_1t-a_0\in K[t;\sigma]$ is irreducible if and only if $z\sigma(z)+a_1z-a_0=0$ has no solutions
in $K$ (e.g. see \cite[(17)]{P66}). If $\sigma$ has order 2 and $a_1\not=2$, then $f(t)=t^2-a_1t-a_0\in F[t]$  is irreducible
in $ K[t;\sigma]$ if and only if it is irreducible in $F[t]$ \cite[(23)]{P66}.
\end{example}

\begin{example}
Let $F=\mathbb{F}_{5}$ and $K=\mathbb{F}_{5^3}$, with $\sigma$ generating the
Galois group of $K/F$. Let
$f(t)=t^2-2\in K[t;\sigma]$. The multiplicative loop $L_f$ of the Hughes-Kleinfeld semifield $S_f$
has order $5^{6}-1=15\,624$ with left and middle nucleus $\mathbb{F}_{5^3}^\times\cdot 1$ and right nucleus
$\mathbb{F}_{25}^\times\cdot 1$. By \cite[Example 16]{W09}, ${\rm Aut}(L_f)$  contains a cyclic subgroup of order $6$.
\end{example}

%
%

\section{Inner automorphisms} \label{sec:inner}

Let $f\in K[t; \sigma]$ have degree $m$, and be monic, irreducible and not right-invariant.

Because of
$|{\rm Inn}(L_f)|=|{\rm Mlt}(L_f)|/(q^{mn}-1),$ the size of the inner mapping group is bounded by
$$q^{mn(mn-1)/2}(q^{mn-1}-1)\cdots(q^2-1) \leq| {\rm Inn}(L_f)|\leq
q^{mn(r-1)/2}(q^{mn-1}-1)\cdots(q-1);$$
cf. Section \ref{subsec:structure}.
In particular, for $q=2$, we get
$|{\rm Inn}(L_f)|=2^{mn(mn-1)/2}(2^{mn-1}-1)\cdots(2^2-1).$
If the semifield $A=K[t;\sigma]/K[t;\sigma]f$ has a nucleus
which is larger than its center $F\cdot 1$, then the inner automorphisms
$\{G_c\,|\, 0\not=c\in {\rm Nuc}(A) \}$
with $G_c(x) = c x c^{-1}$  form a non-trivial subgroup of ${\rm Aut}_F(A)$ \cite[Lemma 2, Theorem 3]{W09} and each such inner automorphism
$G_c$ extends $id_{{\rm Nuc}(A)}$. In particular, the elements in $\{G_c\,|\, 0\not=c\in {\rm Nuc}(A) \}$
 are  middle inner mappings $T_c$ and $\{G_c\,|\, 0\not=c\in {\rm Nuc}(A) \}$
is a non-trivial subgroup of loop isomorphisms in ${\rm Inn}(L_f)$.

 \cite[Proposition 8]{BPS} yields immediately:

\begin{proposition} \label{prop:innerII}
Suppose $N={\rm Nuc}(L_f)\cong\mathbb{F}_{q^l}^\times$ for some integer $1<l\leq n$. Then  $L_f$ has
$$(q^l-1)/(q-1)$$
inner automorphisms which extend to an inner automorphism of $S_f$ and  thus all are middle inner mappings. They are
determined by the $q^l$
elements in its nucleus that do not lie in $F$ and are all extensions of $id_N$.
\\ In particular, if $L_f$ has nucleus $K^\times\cdot 1$ then there are exactly
$$s=(q^n-1)/(q-1)$$
 inner automorphisms of $L_f$ which extend to an inner automorphism of $S_f$ and  thus  are middle inner mappings.
 All
extend $id_{K^\times}$ and have the form $H_{id,k}$ for a suitable $k\in K^\times$.
\end{proposition}

  \cite[Proposition 9]{BPS} yields:

\begin{proposition}\label{prop:estimate}
Let  $n\geq m-1$.
Then there exist at most
$$|{\rm ker}(N_{K/F})|=(q^n-1)/(q-1)$$
 distinct inner automorphisms of $L_f$ that extend to
$H_{id,k}$ such that $N_{K/F}(k) = 1$. These are  middle inner mappings   on $L_f$.
\end{proposition}

Proposition \ref{prop:estimate} and Proposition
\ref{prop:innerII} imply the following estimates for the number of middle inner mappings $T_c$
that are inner automorphisms of $L_f$:

\begin{theorem}\label{cor:innerII}
Let   $n \geq m-1$.
 If $L_f$ has nucleus  $K^\times\cdot 1$ then it has at least $s=(q^n-1)/(q-1)$ inner automorphisms extending $id_{K^\times}$.
 These form a cyclic subgroup of ${\rm Aut}(L_f)$ isomorphic to ${\rm ker}(N_{K/F})$,
 and  a cyclic subgroup of middle inner mappings in ${\rm Inn}(L_f)$ isomorphic to ${\rm ker}(N_{K/F})$.
\end{theorem}

Hence  if  $N={\rm Nuc}(L_f)=\mathbb{F}_{q^l}^\times\cdot 1$, is strictly contained in $K^\times$, $l>1$,  then
  $L_f$ has at least $t$ inner automorphisms extending $id_N$, with
 $$\frac{q^l-1}{q-1}\leq t\leq \frac{q^n-1}{q-1}.$$
 These are  middle inner mappings on $L_f$.

\begin{corollary}\label{cor:inner}
Let   $n \geq m-1$ and assume that $L_f$ has nucleus  $K^\times\cdot 1$. Then
$$|\{T_c\in{\rm Inn}(L_f)\,|\, T_c \text{ automorphism}\}|\geq \frac{q^n-1}{q-1}.$$
\end{corollary}

\begin{example} As in Example \ref{ex:I},
write $K = \mathbb{F}_{p^l},$ so that $q = p^r$ with $l = nr$, and for $r\not=0$ choose
 $\sigma(x)=x^{p^r},$
i.e. ${\rm Fix}(\sigma)=\mathbb{F}_{p^r}=F$ and $\sigma$ has order $n$.
Let $f(t)=t^m-a$. By a straightforward calculation (cf. also \cite[Section 5]{CB}) we obtain immediately:
\\ (i) If $c\in K^\times$ is a primitive $(p^{rm}-1)$th  root of unity, then $G_c$ is an inner automorphism of $L_f$.
\\ (ii) If $p^{gcd(rm,l)}>p^{gcd(r,l)}$, then there exists a non-trivial inner automorphism $G_c$ in $L_f$ for some
$(p^{rm}-1)$th root of unity $c\in K^\times$, which is not a $(p^{r}-1)$th root of unity.
\\ (iii)
Suppose $rm|l$. Then $K$ contains a primitive $(p^{rm}-1)$th root of unity $c$ and ${\rm Aut}(L_f)$ has a cyclic subgroup
$\langle G_c\rangle $ of inner automorphisms
of order $s=(p^{rm}-1)/(p^r-1)$.
\end{example}

%
%

\section{Nonassociative  cyclic algebras} \label{sec:nonasscyclic}

\subsection{} An algebra $S_f$ with $f(t) = t^m - a\in K[t; \sigma]$ irreducible,  $a \in K \setminus F$ and $n\geq m$
 is called a \emph{Sandler semifield} \cite{San62}.
 Since $(t^{m-1}t)t=at$ and $t(t^{m-1}t)=ta=\sigma(a)t$, a Sandler semifield is not $(m+1)$th power-associative.
 For $m=n$, these algebras are also called
\emph{nonassociative cyclic (division) algebras of degree $m$}  and denoted by $(K/F, \sigma, a)$, as they can be seen as canonical
generalizations of associative cyclic algebras.
 We have
$$(K/F,\sigma,a)\cong (K/F,\sigma,b)$$
 if and only if $\sigma^i(a) = kb $ for some $ 0\leq i \leq m-1$ and some $ k \in F^{\times}$ \cite[Corollary 34]{BP}.
 If the elements $1,a,a^2, \ldots,$ $ a^{m-1}$ are linearly
independent over $F$ then  $(K/F, \sigma, a)$ is a  semifield.
In particular, if $m$ is  prime then every  $a \in K \setminus F$ yields a semifield  $(K/F, \sigma, a)$.

In this section  let $a\in K\setminus F$ and
$f(t)=t^m-a\in K[t;\sigma]$
be irreducible (i.e. $a$ does not lie in any proper subfield of $K/F$), $\sigma$ have order $m$ and
 $$A=(K/F,\sigma,a)=K[t;\sigma]/K[t;\sigma](t^m-a).$$
Then
${\rm Nuc}_l(A^\times)={\rm Nuc}_m(A^\times)={\rm Nuc}_r(A^\times)=K^\times \cdot 1.$
 $A^\times$ has exactly $s=(q^m-1)/(q-1)$ inner
automorphisms, all of them extending $id_K$. These are given by the
 $F$-automorphisms $H_{id,l}$   for all $l\in K$ such that $N_{K/F}(l)=1$.
  The subgroup they generate is cyclic and isomorphic to ${\rm ker}(N_{K/F})$. Moreover,
${\rm SL}(m^2,q)\leq {\rm Mlt}(A^\times)\leq {\rm GL} (m^2,q)$
implies
$$q^{m^2(m^2-1)/2}(q^{m^2}-1)(q^{m^2-1}-1)\cdots(q^2-1)$$ $$ \leq| {\rm Mlt}(A^\times)|\leq (q^{m^2}-1)
q^{m^2(m^2-1)/2}(q^{m^2-1}-1)\cdots(q-1),$$
and
$$q^{m^2(m^2-1)/2}(q^{m^2-1}-1)\cdots(q^2-1) \leq| {\rm Inn}(A^\times)|\leq
q^{m^2(m^2-1)/2}(q^{m^2-1}-1)\cdots(q-1).$$
In particular, for $q=2$,
$${\rm Mlt}(A^\times)=2^{m^2(m^2-1)/2}(2^{m^2}-1)(2^{m^2-1}-1)\cdots(2^2-1),$$
$$ {\rm Inn}(A^\times)=2^{m^2(m^2-1)/2}(2^{m^2-1}-1)\cdots(2^2-1).$$
 By Corollary \ref{cor:inner} we know that
 $ |\{T_c\in{\rm Inn}(A^\times)\,|\, T_c \text{ automorphism}\}|\geq s.$
In other words:

\begin{proposition} \label{prop:t^m-a_automorphism_field finite}  \cite[Proposition 16]{BPS}
Let $\alpha$ be a primitive element of $K$, i.e. $K^{\times} =\langle\alpha\rangle$.  Then
$\langle G_{\alpha}\rangle\subset {\rm Aut}(A^\times)$
 is a cyclic subgroup of inner automorphisms
of order $s =(q^m-1)/(q-1)$.
\end{proposition}

\begin{corollary}
For every prime number $m$ there is a loop $L$ of order $q^{m^2}-1$ with center $\mathbb{F}_q^\times\cdot 1$,
 left, middle and right nucleus
$\mathbb{F}_{q^m}^\times\cdot 1$ and a non-trivial automorphism group with a cyclic subgroup of inner automorphisms of order $s$.
\end{corollary}

\begin{proof}
$L$ is the multiplicative loop of a nonassociative cyclic division algebras of degree $m$ over
$\mathbb{F}_q$.
\end{proof}

 \begin{example}
There can be  non-isomorphic Jha-Johnson semifields $S_f$ of a given order
$q^{m^2}$ with different nuclei. For example consider $q = 2$ and $n=m= 4$. Then for any $a$ not contained in any proper subfield of
$\mathbb{F}_{2^4}/\mathbb{F}_2$,
we have the nonassociative cyclic algebra $A=(\mathbb{F}_{2^4}/\mathbb{F}_2,\sigma,a)$ of degree 4.
$A^\times$ has $2^{16}-1$
elements, left, right and middle nucleus $\mathbb{F}_{2^4}^\times\cdot 1$ and center
$\mathbb{F}_2^\times\cdot 1$,
$$|{\rm Mlt}(A^\times)|=2^{16(16-1)/2}(2^{16}-1)(2^{15}-1)\cdots(2^2-1),$$
$$|{\rm Inn}(A^\times)|=2^{16(16-1)/2}(2^{15}-1)\cdots(2^2-1),$$ and
$ |\{T_c\in{\rm Inn}(L_f)\,|\, T_c \text{ automorphism}\}|\geq 15.$
Alternatively, we can  construct the multiplicative loop of a semifield with
$2^{16}-1$ elements choosing $q=16$ and $n=m=2$ (i.e., of a nonassociative quaternion algebra
$A=(\mathbb{F}_{2^8}/\mathbb{F}_{16},\sigma,a)$).
In this case the left, middle and right nucleus is $\mathbb{F}_{2^8}^\times\cdot 1$ and the center
$\mathbb{F}_{16}^\times\cdot 1$.
Then we get
$$16^{6}(16^4-1)(16^{3}-1)\cdots(16^2-1)\leq |{\rm Mlt}(A^\times)|\leq
16^{6}(16^4-1)(16^{3}-1)\cdots(16^2-1)15,$$
$$16^{6}(16^{3}-1)\cdots(16^2-1)\leq |{\rm Inn}(A^\times)|\leq
16^{6}(16^{3}-1)\cdots(16^2-1)15,$$
and $ |\{T_c\in{\rm Inn}(L_f)\,|\, T_c \text{ automorphism}\}|\geq 17.$
\end{example}

We obtain an upper bound on the number of non-isomorphic loops arising from
the nonassociative cyclic algebras of a fixed degree $m$  from  \cite[Theorem 31]{BPS}:

\begin{theorem} \label{numb}
Let $[K:F]=m$.
\\ (i)  If $m$  does not divide $q-1$
 then there are at most
\[\frac{q^m-q}{m(q-1)}\]
non-isomorphic multiplicative loops of order $q^{m^2}-1$ arising from nonassociative cyclic algebras $(K/F, \sigma, a)$ of degree $m$.
\\ (ii)
 If $m$ divides $q-1$ and is prime then there are at most
\[m-1 + \frac{q^m-q - (q-1)(m-1)}{m(q-1)}\]
non-isomorphic multiplicative loops of order $q^{m^2}-1$ arising from nonassociative cyclic algebras $(K/F, \sigma, a)$ of degree $m$.
\end{theorem}

\subsection{The automorphism groups}

 In this subsection, we assume that  $F$ is a field where $m$ is coprime to the
characteristic of $F$, and that $F$ contains a primitive $m$th root of unity
$\omega$, so that $K = F(d)$  where $d$ is a root of some $t^m - c \in F[t]$. Let $s=(q^m-1)/(q-1)$.
The following results are implied by the theorems on the automorphism groups for the related algebras $S_f$ \cite[Theorems 19, 20]{BPS}:

\begin{theorem} \label{thm:semidirect}
Let $S_f = (K/F,\sigma,a)$  where $a = \lambda d^i$ for some
$i \in \{ 1 , \ldots, m-1 \}$, $\lambda \in F^{\times}$.
\\ (a)
Suppose $m$ is odd or $(q-1)/m$ is even.  Then $\mathrm{Aut}(L_f)$
 contains a subgroup isomorphic to the semidirect product
\begin{equation} \label{eqn:Automorphisms of nonassociative cyclic algebras over finite fields semidirect not nec prime}
\mathbb{Z} / \Big( \frac{s}{m} \Big) \mathbb{Z} \rtimes_{q} \mathbb{Z} / (m \mu) \mathbb{Z},
\end{equation}
where $\mu = m/\mathrm{gcd}(i,m)$.
Its automorphisms extend to automorphisms of $S_f$.
\\ (b) Suppose $m$ is prime and divides $ q-1$.
\begin{itemize}
\item[(i)] If $m = 2$ then $\mathrm{Aut}(L_f)$ contains a subgroup isomorphic to  the dicyclic group $\mathrm{Dic}_l$ of order $4l = 2q + 2$.
\item[(ii)] If $m > 2$ then $\mathrm{Aut}(L_f)$ contains a subgroup isomorphic to  the semidirect product
\begin{equation} \label{eqn:Automorphisms of nonassociative cyclic algebras over finite fields semidirect}
 \mathbb{Z} / \Big( \frac{s}{m} \Big) \mathbb{Z} \rtimes_{q} \mathbb{Z} / (m^2) \mathbb{Z}.
\end{equation}
\end{itemize}
The automorphisms in these subgroups extend to automorphisms of $S_f$.
\end{theorem}

\begin{corollary}
Let $F$ have characteristic not $2$,
$K/F$ be a quadratic field extension
and $S_f=(K/F, \sigma, a)$ a nonassociative quaternion algebra.
\\ (i) If
 $a\not=\lambda d$
 for any  $\lambda\in F^\times$, then ${\rm Aut}(L_f)$ contains a subgroup isomorphic to the cyclic group
 $\mathbb{Z}/(q+1) \mathbb{Z}.$
  All the automorphisms in this subgroup are inner and extend to automorphisms of $S_f$.
\\ (ii) If  $a=\lambda d$ for some  $\lambda\in F^\times$, then
${\rm Aut}(L_f)$ contains  a subgroup isomorphic to
the dicyclic group of order $2q+2$ and all its elements extend to automorphisms of $S_f$.
\end{corollary}

Note that nonassociative quaternion algebras (where $m=2$ in  Theorem \ref{thm:semidirect}) are up to isomorphism the only
  proper semifields of order $q^4$ with center $F\cdot 1$ and nucleus containing the quadratic field extension $K$ of $F$
  (embedded into the algebra as $K\cdot 1$ as usual)
\cite[Theorem 1]{W}.
 For the loop $L_f$ of a nonassociative quaternion algebra, moreover
$$q^{6}(q^4-1)(q^{3}-1)(q^2-1)\leq |{\rm Mlt}(L_f)|\leq q^{6}(q^4-1)(q^{3}-1)(q^2-1)(q-1).$$

\begin{example}\label{ex:quat2}
Let $F= \mathbb{F}_2$ and let $K = \mathbb{F}_4$, then
$K = \{0,1,x, 1+x \}$ where $x^2+x+1 = 0$. There is up to isomorphism only
one nonassociative quaternion algebra  which can be constructed
using $K/F$ given by $(K/F,\sigma,x)  $  \cite[Example 30]{BPS}.
This is up to isomorphism the only proper semifield of order $16$ with center $F\cdot 1$ and nucleus containing $K\cdot 1$.

Its $F$-automorphism group consists of inner automorphisms and is
isomorphic to $\langle G_x \rangle\cong \mathbb{Z}/3\mathbb{Z}$. Its multiplicative loop $L$ thus has 15
elements, left, right and middle nucleus
$ \mathbb{F}_4^\times\cdot 1$, center $ \mathbb{F}_2^\times\cdot 1$ and non-trivial automorphism group containing
$\langle G_x \rangle$. $L$ is right cyclic. The inner mapping group of $L$ contains a
cyclic subgroup of inner automorphisms
isomorphic to $\mathbb{Z}/3\mathbb{Z}$ induced by the inner automophisms of the algebra.
We know
$|{\rm Inn}(L)|=
1\,344$ and
$|{\rm Mlt}(L)|=
20\,160.$
\end{example}

\begin{example} \label{ex:quat3}
Let $F = \mathbb{F}_3$ and $K = \mathbb{F}_9$, i.e. $K = F[x]/(x^2 - 2) = \{0,1,2,x,2x,$
$ x+1, x+2, 2x+1, 2x+2\}.$
There are exactly two non-isomorphic semifields which are nonassociative quaternion algebras with nucleus
$K\cdot 1$, given by
$A_1 = (K/F, \sigma, x)$ with ${\rm Aut}_F(A_{1})\cong \mathbb{Z}/4 \mathbb{Z}$  and by
$A_{2}= (K/F, \sigma, x+1)$ with
${\rm Aut}_F(A_2)$  isomorphic to the group  of quaternion units,
the smallest dicyclic group ${\rm Dic}_2$ (of order 8)   \cite[Example 32]{BPS}.
These are up to isomorphism the only two proper semifields of order $81$ with center $F \cdot 1$ and nucleus containing $K\cdot 1$
\cite[Theorem 1]{W}. They are listed as cases (X) and (XI) in \cite{Dem}.

 Their multiplicative loops $L_1$ and $L_2$ both have order 80,
 left, middle and right nucleus $\mathbb{F}_9^\times\cdot 1$ and center $\mathbb{F}_3^\times\cdot 1$.
 Both are left and right cyclic. Their multiplicative mapping group has
$ 12 \,130 \,560$
 elements and their inner mapping group has
$151 \,632$ elements.
 The  automorphism group of $L_1$ and the  inner mapping group of $L_1$ contain
  a subgroup of inner automorphisms which is isomorphic to $\mathbb{Z}/4 \mathbb{Z}$, while
the automorphism group of $L_2$ contains a subgroup of order 8 isomorphic  to the group  of quaternion units
 ${\rm Dic}_2$, all of whose elements can be extended to automorphisms of $S_f$.
\end{example}

\begin{example} \label{ex:quat4}
Let $F = \mathbb{F}_5$ and $K = \mathbb{F}_{25}\cong\mathbb{F}_5(\sqrt{2})$. Then
 $$(\mathbb{F}_5(\sqrt{2})/\mathbb{F}_5,\sigma,\sqrt{2})$$ is a
 nonassociative quaternion division algebra with $|{\rm Aut}(L_f)|\geq 12$
 and $$(\mathbb{F}_5(\sqrt{2})/\mathbb{F}_5,\sigma,1+2\sqrt{2})$$ is a
 nonassociative quaternion division algebra with $|{\rm Aut}(L_f)|\geq 6$. The algebras are non-isomorphic
 \cite[Example 13]{W09}.
 Their multiplicative loops have order $5^4-1=624$, left, middle and right nucleus $\mathbb{F}_{25}^\times\cdot 1$
 and center $\mathbb{F}_5^\times\cdot 1$.
  Their multiplicative mapping group has
$29\,016\,000\,000$
 elements and their inner mapping group
$ 46\,500\,000$ elements.
\end{example}

%
%

\section{Isotopies}\label{sec:isotopies}

Two semifields $A$ and $A'$ are \emph{isotopic} if there are linear bijective maps $F,G,H:A\longrightarrow A'$ such that
for all $x,y\in A$, we have $F(x)G(y)=H(xy)$. The triple $(F,G,H)$ is called an \emph{isotopism}. Any
 isotopy of semifields canonically induces an isotopy of the multiplicative loops $L$ and $L'$ of $A$ and $A'$.

Note that if $f$ and $g$ are \emph{similar} irreducible monic polynomials of degree $m$, i.e.
$gu \equiv 0\text{ mod }f$ for some $0\not=u\in R_m$, then $L_f$ and $L_g$ are isotopic loops by \cite[Theorem 7]{LS}.

Indeed, Kantor showed that there
 are less than $r\sqrt{log_2(r)}$  non-isotopic semifields of order $r$ that can be
 obtained through our construction \cite{K}, so that there are less than $r\sqrt{log_2(r)}$  non-isotopic loops of
 order $r-1$ obtained as their multiplicative loops.
By \cite[Section 8]{LS}, we have some more upper bounds on the number of isotopic loops which can be obtained as
 multiplicative loops of Jha-Johnson semifields:

 Let $N(q,m)=|I(q,m)|$ where
 $$|I(q,m)|=\frac{1}{m}\sum_{l/m}\mu(l) q^{m/l}=\frac{q^m-\theta}{m}$$
  is the number of monic irreducible polynomials of degree $m$ in the center of $R$ which is
  $\mathbb{F}_q[y]\cong \mathbb{F}_q[t^n;\sigma]$. Here, $\mu$ is the Moebius function, and
  $\theta$ the number of elements of $\mathbb{F}_{q^m}$ contained in a proper subfield  $\mathbb{F}_{q^e}$
  of  $\mathbb{F}_{q^m}$.
 Let $A(q,n,m)$ denote the number of isotopy classes of semifields $A=K[t;\sigma]/K[t;\sigma]f(t)$ of order $q^{nm}$
 defined using $f(t)\in K[t;\sigma]$
 as in the previous sections, with
 $$(|C(A)|,|{\rm Nuc}_l(A)|,|{\rm Nuc}_m(A)|,|{\rm Nuc}_r(A)| )=(q,q^n,q^n,q^m).$$
 Then
 $$A(q,n,m)\leq N(q,m)=\frac{q^m-\theta}{m}$$
 and, even stronger,
 $$A(q,n,m)\leq M(q,m)$$
 with $M(q,d)$ being the number of orbits in $I(q,m)$ under the action of the group
 $$G=\Gamma L(1,q)=\{(\lambda,\rho)\,|\,\lambda\in\mathbb{F}_q^\times,\rho\in{\rm Aut}(\mathbb{F}_q)\}$$
 on $I(q,m)$ defined via
 $$f^{(\lambda,\rho)}(y)=\lambda^{-m}f^\rho(\lambda y).$$

 For $q=p^r$ we have
 $$\frac{q^m-\theta}{mr(q-1)}\leq  M(q,m)\leq \frac{q^m-\theta}{m}$$
\cite{LS}.
These are also upper bounds for the number of non-isotopic multiplicative loops $A^\times$
of order $q^{nm}-1$ defined using $A$.

\begin{example}
(i) For $q=2,3,4,5$ and $n=m=2$ we have the tight bound $M(q,m)=1,2,1,3$ \cite{LS} and so there is up to isotopy exactly one
loop $L$ with
 $$(|C(L)|,|{\rm Nuc}_l(L)|,|{\rm Nuc}_m(L)|,|{\rm Nuc}_r(L)| )=(q-1,q^2-1,q^2-1,q^2-1)$$
 that is the multiplicative loop of a Jha-Johnson semifield when $q=2,4$, and at most 2 (respectively 3) loops, when
$q=3$ (respectively $5$).

When $q=2$ this loop is given by the invertible elements in the nonassociative quaternion algebra in Example \ref{ex:quat2}.
When $q=3$ (respectively, $5$) there are the two non-isomorphic, and hence necessarily also non-isotopic, nonassociative quaternion
algebras  in Examples \ref{ex:quat3} and \ref{ex:quat4}.
It is not clear if the corresponding two multiplicative loops are isotopic or not.
\\ (ii) If $q=p$ and $gcd(p-1,m)=1$ then
$$M(p,m)=\frac{N(p,m)}{(p-1)}=\frac{q^m-\theta}{m(p-1)}$$
and so there are at most $(q^m-\theta)(m(p-1))$ non-isotopic loops $L$ of order $p^{mn}-1$
with
 $$(|C(L)|,|{\rm Nuc}_l(L)|,|{\rm Nuc}_m(L)|,|{\rm Nuc}_r(L)| )=(p-1,p^n-1,p^n-1,p^m-1)$$
which  arise as the multiplicative loop of a Jha-Johnson semifield.
\end{example}

We conclude by noting that it would be interesting to also obtain lower bounds for the number of isotopy classes of loops arising from this
construction.



\end{document}